\documentclass[11pt]{amsart}

\usepackage{amssymb, amsmath, amsthm, color, hyperref, url, fullpage}

\newtheorem{theorem}{Theorem}[section]
\newtheorem{corollary}[theorem]{Corollary}
\newtheorem{lemma}[theorem]{Lemma}
\newtheorem{prop}[theorem]{Proposition}

\theoremstyle{definition}
\newtheorem{defn}[theorem]{Definition}
\newtheorem{example}[theorem]{Example}

\begin{document}

\title{A Note on Rickart Modules}

\author{ Ali H. Al-Saedi}
\address{Ali H. Al-Saedi, Department of Mathematics, Oregon State University, Corvallis, OR 97331, USA.}
\email{alsaedia@math.oregonstate.edu.}

\author{Mehdi S. Abbas}
\address{Mehdi S. Abbas, Department of Mathematics, College of Science, Al-Mustansiriyah University, Baghdad, Iraq.}
\email{m.abass@uomustansiriyah.edu.iq.}

\begin{abstract}
We study the notion of Rickart property in a general module theoretic setting as a generalization to the concept of Baer modules and right Rickart rings. A module $M_{R}$ is called Rickart if the right annihilator in $M_{R}$ of each left principal ideal of $End_{R}(M)$ is generated by an idempotent. Characterizations of Rickart modules are given. We give sufficient conditions for direct sums of Rickart modules to be Rickart. We extend some useful results of right Rickart rings to the theory of Rickart modules. 
\end{abstract}

\keywords{Baer modules and rings, Rickart modules and rings}

\subjclass[2010]{16D10, 16D40, 16D80.}

\maketitle


\section{Introduction}

Throughout this paper, $R$ will denote an associative ring with identity. The modules are all unital right modules and denoted by $M$, and the endomorphism ring by $S = End_{R}(M)$. In $1960$, Maeda \cite{maeda} defined Rickart rings in an arbitrary setting. A ring $R$ is called a \textit{right} (resp., \textit{left}) \textit{Rickart} ring if the right (resp., left) annihilator of any element in $R$ is generated by an idempotent as a right (resp., left) ideal of $R$. A ring $R$ is called a \textit{Rickart} ring if it is a right and left Rickart ring. A right (resp., left) Rickart ring has a homological characterization as right (resp., left) p.p. ring which means every principal right (resp., left) ideal is projective. $R$ is called \textit{von Neumann} regular if for each element $a$ in $R$ there exists $x$ in $R$ such that $axa = a$. A ring $R$ is called \textit{right} (resp., \textit{left}) \textit{semihereditary} if every finitely generated right (resp., left) ideal of $R$ is projective. In \cite{lee}, the definition of right Rickart rings is generalized to the module theoretic setting by G. Lee, S.T. Rizvi and C. Roman. A module $M_{R}$ is called \textit{Rickart} if the right annihilator of each left principal ideal of $End_{R}(M)$ is generated by an idempotent, i.e, for each $\varphi \in S$, there exists $e=e^{2}$ in $S=End_{R}(M)$ such that $r_{M}(\varphi)= eM$.  Kaplansky, in his writing ``\textit{Rings of operators}"\cite{kaplansky}, defined the concept of Baer rings. A ring $R$ is celled a \textit{right} \textit{Baer} ring if every right annihilator in $R$ is of the form $eR$ for some idempotent $e$ in $R$. \textit{left} \textit{Baer} rings are defined similarly \cite{kaplansky}. It is easy to see that right and left Baer rings are symmetric. In \cite{cosmin}, C.S. Roman and S.T. Rizvi generalized the definition of Baer rings to the scope of module theory. An $R$-module $M$ is called a \textit{Baer} module if for each left ideal $I$ of $S = End_{R}(M)$, $r_{M}(I) = eM$ for $e^{2} = e$ in $S$, equivalently, for each submodule $N$ of $M$, $l_{S}(N) =  Se$ for $e^{2} = e$ in $S$.  A module $M$ is said to have the \textit{summand intersection property} , shortly SIP, if the intersection of any two direct  summands of $M$ is a direct summand \cite{wilson}. An element $m$ in $M$ is called a \textit{singular} element of $M$ if the right ideal $r_{R}(m)$ is essential in $R_{R}$. The set of all singular elements of $M$ is denoted by $Z(M)$. $M_{R}$ is called a \textit{nonsingular} (resp., \textit{singular}) module if $Z(M)$ = 0 (resp., $Z(M) = M$). In particular, $R$ is called  \textit{right} \textit{nonsingular} if $Z(R_{R}) = 0$. \textit{Left} \textit{nonsingular} rings are defined similarly, and  \textit{nonsingular} ring  shall mean a ring that is both right and left nonsingular \cite{lam}. A module $M_{R}$ is called $K$-\textit{nonsingular} if, for all $\varphi$ in $S$, $r_{M}(\varphi) = Ker(\varphi)$ is essential in $M$ implies $\varphi = 0$ \cite{cosmin}. A module $ M_{R} $ is called \textit{retractable} if $ Hom_{R}(M,N)\neq 0$ for all nonzero submodules $N$ of $M$  ,  (or, equivalently, there exists $\varphi\in S$ with $Im(\varphi)$ is a submodule of $N$) \cite{cosmin}. A module $M_{R}$ is called \textit{quasi-injective}, if every submodule  $N$ of $M$ and homomorphism $ f\colon N \rightarrow M $  there exists a homomorphism $ g\colon M \rightarrow M $ such that $ g_{|N}=f $, i.e. $g(n) = f(n)$ for all $n$ in $N$  \cite{lam}. A submodule $N$ of a module $M_{R}$ is called \textit{closed} in $M$ if it has no proper essential extension in $M$ \cite{lam}. A module $M$ is called \textit{extending} if every submodule of $M$ is essential in a direct summand (or, equivalently, each closed submodule of $M$ is a direct summand). A ring   $R$ is called a \textit{right} \textit{extending} if it is extending as right $R$-module[6].


\section{Rickart Modules}
In this section we study the basic properties of Rickart modules. In \cite{lee}, it is shown that a direct summand of a Rickart module is also a Rickart module, but the direct sum of Rickart modules does not necessary inherit the Rickart property. We give sufficient conditions under which  the direct sum of Rickart modules has the Rickart property. As Rickart modules are a proper generalization for Baer modules, we also give a sufficient condition for which Baer and Rickart properties are equivalent.
 
\begin{defn}[\cite{lee}, Definition\;2.2]
Let $M$ be a right $R$-module and let $S = End_{R}(M)$. Then $M$  is called a \textit{Rickart} \textit{module} if the right annihilator in $M$ of any single element of $S$ is generated by an idempotent of $S$. Equivalently, for all $\varphi$ in  $S$,  $r_{M}(\varphi) = Ker(\varphi)= eM$ for some $e^{2}= e$ in $S$. 
\end{defn}

\begin{prop}\label{th1}
A right $R$-module $M$ is Rickart if and only if any two direct summands $A$ and $B$ of $ M $ and each $R$-homomorphism $ f\colon A \rightarrow B $, then  $Ker(f)$ is a direct summand of $A$. 
\end{prop}
\begin{proof}
Let $ f \in  Hom_{R}(A,B)$, then $f$ can be extended to $f^{\prime} = f\pi$ where $\pi$ be the projection map onto $A$.  It is easy to check that $Ker(f^{\prime}) = Ker(f)\oplus C$  for $M=A\oplus C$ where $C$ is a direct summand submodule of $M$. Since $M$ is \textit{Rickart}, then $Ker(f^{\prime}) \leq^{\oplus} M$, and hence $Ker(f)\leq^{\oplus} Ker(f^{\prime})\leq^{\oplus} M$ which implies $Ker(f)\leq^{\oplus} M$. Thus  $M=Ker(f)\oplus D$ for some direct summand submodule $D$ of $M$. Since $Ker(f)\subseteq A$, and by the help of modular law we obtain that $A= A\cap M = Ker(f)\oplus A\cap D$. Therefore, $Ker(f)\leq^{\oplus} A$. Conversely, put $A=B=M$.
\end{proof}

\begin{theorem}[\cite{lee},Theorem\;2.7]\label{th2}
A direct summand of Rickart module is a Rickart module. 
\end{theorem} 
Theorem \ref{th2} follows easily by setting $A=B$ in Proposition \ref{th1} 
\begin{example}
A finite direct sum of Rickart modules is not necessarily a Rickart module. The $\mathbb{Z}$-module $\mathbb{Z}\oplus Z_{2}$ is not Rickart while $\mathbb{Z}$ and $Z_{2}$ are both Rickart $\mathbb{Z}$-modules. Consider $f:\mathbb{Z}\rightarrow Z_{2}$ defined by $f(x)=\bar{x}$ for all  $x\in \mathbb{Z}$. Thus $Ker(f)=2\mathbb{Z}$ is not a direct summand of $\mathbb{Z}$. 
\end{example}
Recall that a module $M$ is called $N$-\textit{Rickart} (or \textit{relatively} \textit{Rickart} \textit{to} $N$) if, for every homomorphism $\varphi:M\rightarrow N$, $Ker(\varphi)\leqslant^{\oplus} M$.
Now, we give sufficient conditions for the direct sums of Rickart modules to be Rickart. 
\begin{theorem}\label{th3}
Let $M_{1}$ and $M_{2}$ be Rickart $R$-modules.\;If we have the following conditions
\begin{itemize}
\item[(1)]$ N\leqslant M_{1}\oplus M_{2}\Rightarrow N=N_{1}\oplus N_{2}$, where $N_{i}\leqslant M_{i}$ for $i=1,2$,
\item[(2)] $M_{i}\; and\; M_{j}\; are\; relatively\; Rickart\; modules\; for\; i,j=1,2,$ 
\end{itemize}
then $M=M_{1}\oplus M_{2}$ is a Rickart module.
\end{theorem}
\begin{proof}
Let $S=End_{R}(M)=End_{R}(M_{1}\oplus M_{2})$ and $\varphi \in S$.\;Then by (1), $Ker(\varphi)=N_{1}\oplus N_{2}$ where $N_{i}\leqslant M_{i}$, $i=1,2$. Let $S_{i}=End_{R}(M_{i})$ for $i=1,2$. Since $S$ can be written in the following matrix form 
$ $ \[S =
 \begin{pmatrix}
  S_{1} & Hom_{R}(M_{2},M_{1})\\
  Hom_{R}(M_{1},M_{2}) & S_{2}\\
 \end{pmatrix}, 
 \] $ $
then set $\varphi=(\varphi_{ij})_{i,j=1,2}$ with $\varphi_{ij}:M_{j}\to M_{i}$.\;Let $N_{1}^{\prime}=r_{M_{1}}(S_{1}\varphi_{11})=Ker(\varphi_{11})$. We now show that $N_{1}=N_{1}^{\prime}\cap Ker(\varphi_{21})$.  We note that 

\begin{align*}
x \in N_{1} \iff & (x,0)\in N_{1}\oplus N_{2}=Ker(\varphi) \\
\iff & \begin{pmatrix}
 \varphi_{11} & \varphi_{12}\\
 \varphi_{21} & \varphi_{22}\\
 \end{pmatrix}
 \begin{pmatrix}
  x\\0
 \end{pmatrix}
 =
 \begin{pmatrix}
 0\\0
 \end{pmatrix} \\
\iff &  \;\varphi_{11}(x)=\varphi_{21}(x)=0 \\
\iff &  \;x\in N_{1}^{\prime}\; and\; x\in Ker(\varphi_{21}) \\
\iff & \;x\in N_{1}^{\prime}\cap Ker(\varphi_{21}) \\
\iff & \;N_{1}=N_{1}^{\prime}\cap Ker(\varphi_{21}).
\end{align*}
By(2), $Ker(\varphi_{21})\leqslant^{\oplus}M_{1}$. By [\cite{lee},Proposition 2.16], $M_{1}$ has SIP, thus $N_{1}=N_{1}^{\prime}\cap Ker(\varphi_{21})\leqslant^{\oplus}M_{1}$. Similarly, $N_{2}=N_{2}^{\prime}\cap Ker(\varphi_{12})\leqslant^{\oplus}M_{2}$. This implies $r_{M_{1}\oplus M_{2}}(\varphi)=N_{1}\oplus N_{2}\leqslant^{\oplus} M_{1}\oplus M_{2}$. 

\end{proof}

\begin{corollary}\label{th6}
Let $M_{1}$ and $M_{2}$ be Rickart $R$-modules. If we have the following conditions
\begin{itemize}
\item[(1)]$r_{R}(M_{1})+r_{R}(M_{2})=R,$
\item[(2)]$M_{i}\; and\; M_{j}\; are\; relatively\; Rickart\; modules\; for\; i,j=1,2,$ 
\end{itemize}
then $M=M_{1}\oplus M_{2}$ is a Rickart module.
\end{corollary}

\begin{proof}
Since $r_{R}(M_{1})+r_{R}(M_{2})=R$, then by [\cite{mehdi}, Proposition 4.2] any submodule $N$ of $M_{1}\oplus M_{2}$ can be written in the form $N=N_{1}\oplus N_{2}$ for $N_{1}\leqslant M_{1}$ and $N_{2}\leqslant M_{2}$. 
\end{proof}

\section{Endomorphism Rings}
There is an interesting connection between Rickart modules and their endomorphism rings. In fact, Rickart modules provide a source for Rirckart rings by considering their endomorphism rings. This connection makes the definition of Rickart modules more valuable. 

\begin{prop}[\cite{lee}, Proposition 3.2]\label{th7} 
If $M$ is a Rickart $R$-module then $S=End_{R}(M)$ is a right Rickart ring.
\end{prop}
On the other hand, the fact that the endomorphism ring of a module is Rickart does not imply that the module itself is Rickart module.
\begin{example}
Let $M=\mathbb{Z}_{p^{\infty}}$, considered as $\mathbb{Z}$-module where $p$ is a prime number. It is well known that $End_{\mathbb{Z}}(M)$ is the ring of p-adic integers (\cite{fuchs}, example 3, page 216). Since the ring of p-adic integers is a commutative domain, it is Rickart. However, $M=\mathbb{Z}_{p^{\infty}}$ is not Rickart (see \cite{lee}, example 2.17).
\end{example}
Recall that a module $M_{R}$ is said to be retractable if $Hom_{R}(M,N)\neq 0$, for all $0\neq N\leqslant M$(or, equivalently, there exists $\varphi \in S$ with $Im(\varphi)\subseteq N)$ \cite{cosmin}. For a module $M$ being retractable, it gives the equivalency of Rickart property for the module $M$ and its endomorphism ring $S=End_{R}(M)$ in the following theorem.
\begin{prop}[\cite{lee}, Proposition 3.5]
Let $M$ be a (quasi-)retractable module. Then the following conditions are equivalent:
\begin{itemize} 
\item[(a)] $M$ is a Rickart module;
\item[(b)] $End_{R}(M)$ is a right Rickart ring.
\end{itemize}
\end{prop}
Recall that a module $M_{R}$ is said to be $k$-\textit{local}-\textit{retractable} if for any $\varphi\;\in End_{R}(M)$ and any nonzero element $m\in r_{M}(\varphi)=Ker(\varphi)\subseteq\;M$, there exists a homomorphism $\psi_{m}:M\rightarrow r_{M}(\varphi)$ such that $m\in \psi_{m}(M)\subseteq\;r_{M}(\varphi)$\cite{lee}. Every Rickart module is $k$-local-retractable, see [\cite{lee}, Proposition 3.7], but the converse is not always true, for example, The $\mathbb{Z}$-module $\mathbb{Z}_{4}$ is $k$-local-retractable but it is not Rickart.  
\begin{theorem}[\cite{lee}, Theorem 3.9]
The following condition are equivalent:
\begin{itemize}
\item[(a)]$M$ is a Rickart module.
\item[(b)]$End_{R}(M)$ is a right Rickart ring and $M$ is $k$-local-retractable.
\end{itemize}
\end{theorem} 
 
\begin{prop}
Let $M$ be an $R$-module such that its endomorphism ring is a regular. Then $M$ is Rickart, and consequently $S=End_{R}(M)$ is a right Rickart ring. If $M$ is Rickart and self-cogenerator module then $S$ is regular. 
\end{prop}
\begin{proof}
Take $\varphi\;\in\;S$. Since $S$ is regular, there exists $\psi\;\in\;S$ so that $\varphi=\varphi\psi\varphi$, thus $\psi\varphi=(\psi\varphi)(\psi\varphi)$ is an idempotent with the property that $S\varphi=S\psi\varphi$, then $r_{M}(S\varphi)=r_{M}(\psi\varphi)=(1-\psi\varphi)M\leqslant^{\oplus}M$. Hence $Ker(\varphi)\leqslant^{\oplus}M$.
Conversely, assume that $M$ is Rickart and self-cogenerator. Let $\varphi\in S$ , then $Ker(\varphi)\leqslant^{\oplus}M$ since $M$ is Rickart. By Proposition \ref{th7}, we have $S$ is a right Rickart ring. Since $M$ is self-cogenerator then by [\cite{wisbauer}, Proposition 39.11(4), page 335] we obtain $Im(\varphi)\leqslant^{\oplus}M$ and hence $S$ is regular. 
\end{proof}
We recall the following chart of basic implications.
\begin{table}[h]
\begin{tabular}{lllll}
 &&(Baer ring) &\\ 
 &&$\;\;\;\;\;\;\;\;\Downarrow$ &  \\
$(von\; Neumann\; regular)\Rightarrow$&$(right\; semihereditary)\Rightarrow$&$(right\; Rickart)\Rightarrow$&$(right\;nonsingular)$&  \\
\end{tabular}
\end{table}

We now show that all of the implications are irreversible. 
\begin{example}
Let $R$ be a commutative reduced ring which has exactly two idempotents $0$ and $1$, but which is not an integral domain. By [\cite{lam2}, exercise 7.23], if $R$ is a ring with exactly two idempotents $0$ and $1$, then $R$ is Baer if and only if  $R$ is Rickart if and only if $R$ is a domain. Then $R$ is not Rickart. More explicitly, take $R=\mathbb{Q}[a,a^{\prime}]$ with the relation $a a^{\prime}=0$. Using unique factorization  in the polynomial ring $R=\mathbb{Q}[x,x^{\prime}]$, we verify easily that $R$ is reduced, with $0$ and $1$ as its only idempotents. Since $a a^{\prime}=0$ and $a, a^{\prime}\neq 0$, $R$ is not integral domain and hence $R$ is not Rickart. Here the principal ideal $aR$ is not a projective and $r_{R}(a)=a^{\prime}R$ is not generated by an idempotent. By [\cite{lam}, Corollary 7.12], A commutative ring is nonsingular if and only if it is reduced. We conclude that $R$ is nonsingular which is not Rickart.
\end{example}
\begin{example}
Let $R=\mathbb{Z}[x]$. Since $\mathbb{Z}$ is not a von Neumann regular ring, then by \cite{victor}, $R$ is not semihereditary. However, $\mathbb{Z}$ is a reduced Rickart ring, then by Armendariz  [\cite{armendariz}, Theorem A], $R=\mathbb{Z}[x]$ is a Rickart ring. Thus $R$ is a Rickart ring that is not semihereditary.
\end{example}
\begin{example}
$\mathbb{Z}$ is a semihereditary ring but it is not von Neumann regular.  
\end{example}
\begin{example}\label{ex1}
Consider the ring $A=\prod_{i=1}^{\infty}\;\mathbb{Z}_{2}$ and let $R=\{(a_{n})_{n=1}^{\infty}\mid\; a_{n}\; \mbox{is eventually constant}\}$, is a subring of $A$, then $R$ is von Neumann regular ring, and hence it is Rickart but it is not Baer. To show $R$ is not Baer, let $e_{i}=(e_{i,j})$ denote the $i^{th}$ unit vector in $R$, that is, $e_{i,j}=1$ for $j=i$ and $e_{i,j}=0$ for $j\neq i$. Also, let $I$ be the ideal of $R$ generated by the set $\{e_{i} \mid i\;\mbox{is\;odd}\}$. Clearly, $r_{R}(I)$ consists of $e_{i}$ for $i$ even and sequences $a=(a_{n})$ which are eventually zero and $a_{n}=0$ for $n$ odd. If  $r_{R}(I)=aR$ for some $a=(a_{n})$ in $r_{R}(I)$, then there exists an integer $k$ such that $a_{n}=0$ for all $n\geq k$. Since $e_{2k}$ in $r_{R}(I)$, then there exists an element $r=(r_{n})$ in $R$ such that $e_{2k}=ar$. Thus $1=e_{2k,2k}=a_{2k} r_{2k}=0 r_{2k}=0$ which is a contradiction. Therefore, $r_{R}(I)$ cannot be generated by any element in $R$ and hence cannot be generated by any idempotent.  
\end{example}

We recall a simple known result.
\begin{lemma}\label{lemma}
A ring $R$ is right nonsingular if and only if the right annihilator of any subset $X$ of $R$ is always a closed right ideal.
\end{lemma}
In \cite{osofsky}, B.L. Osofsky showed that the endomorphism ring of a quasi-injective module is right self-injective ring hence it is right extending ring. In \cite{faith}, It was proved by Faith and Utumi that if $M$ is a quasi-injective module and $S = End_{R}(M)$, then $J(S)$ consists of all endomorphisms of $M$ having essential kernel and $S/J(S)$ is a von Neumann regular ring. 
We obtain the following theorem which is an extending to [\cite{lam}, Theorem 7.25(1), p. 262]. 
\begin{theorem}
Let $M$ be a quasi-injective $R$-module and $S=End_{R}(M)$. The following statements are equivalent:
\begin{itemize}
\item[(1)]$M$ is a Baer module.
\item[(2)]$M$ is a Rickart module.
\item[(3)]$S$ is a von Neumann regular ring.
\item[(4)]$S$ is a right semihereditary ring. 
\item[(5)]$S$ is a right Rickart ring.
\item[(6)]$S$ is a right nonsingular ring. 
\end{itemize}
\end{theorem}
\begin{proof}
$(1)\Rightarrow(2)$ It is clear.\newline
$(2)\Rightarrow(3)$ Since $M$ is quasi-injective, then by [\cite{faith} or \cite{lam}, Theorem 13.1, p.\;359], the Jacobson radical of $S$, $J(S)=\{f\in S|Ker(f)\leqslant^{e}M\}$ and $S/J(S)$ is a von Neumann regular ring. Since $M$ is Rickart then for each $f\in S$ implies $Ker(f)\leqslant^{\oplus}M$. Thus we have for all $f\in J(S)$ implies $f=0$ so $J(S)=0$ and hence $S/J(S)=S/0\simeq S$. Therefore, $S$ is a von Neumann regular ring.\newline
$(3)\Rightarrow(4)\Rightarrow(5)\Rightarrow(6)$ It is clear.\newline
$(6)\Rightarrow(1)$ Let $N$ be a submodule  of $M$. By Lemma \ref{lemma}, $r_{S}l_{S}(N)$ is closed right ideal in $S$. Since $M$ is quasi-injective, then by \cite{osofsky}, $S$ is a right self-injective ring and hence it is a right extending ring. Thus $r_{S}l_{S}(N)=eS$ for some $e=e^{2}\in S$. Therefore, $l_{S}(N)=l_{S}(r_{S}(l_{S}(N)))=S(1-e)$, thus $M$ is Baer.
\end{proof}
\begin{corollary}[\cite{lam}, Theorem 7.52(1), p. 262]
Let $R$ be any right self-injective ring. Then the following statements are equivalent:
\begin{itemize}
\item[(1)]$R$ is a Baer ring.
\item[(2)]$R$ is a von Neumann regular ring.
\item[(3)]$R$ is a right semihereditary ring. 
\item[(4)]$R$ is a right Rickart ring.
\item[(5)]$R$ is a right nonsingular ring. 
\end{itemize}
\end{corollary}
L. Small proved if a ring $R$ has no infinite set of nonzero orthogonal idempotents, then $R$ is a Baer ring if and only if $R$ is a right Rickart ring. Example \ref{ex1} is an illustrated example which shows there is no equivalency between Baer property and Rickart property if the ring has infinite set of nonzero orthogonal idempotents.    

In his Master thesis 2008 \cite{ali}, Ali Al-Saedi extended Small's result to the general module theoretic settings as the following: Let $M$ be a right $R$-module such that $S=End_{R}(M)$ has no infinite set of non-zero orthogonal idempotents. Then $M$ is a Baer module if and only if $S$ is a right Rickart ring.
Independently in \cite{lee}, G. Lee, T. Rizvi, and C. Roman  proved an analogue result as the following: Let $M$ be a right $R$-module and let $S=End_{R}(M)$ have no infinite set of non-zero orthogonal idempotents. Then $M$ is a Baer module if and only if $M$ is a  Rickart module. We combine those two results to get the next theorem.
\begin{theorem}
Let $M$ be a right $R$-module such that $S=End_{R}(M)$ has no infinite set of non-zero orthogonal idempotents. Then the following statements are equivalent:
\begin{itemize}
\item[(1)]$M$ is a Baer module.
\item[(2)]$M$ is a Rickart module.
\item[(3)]$S$ is a right Rickart ring.
\end{itemize}
\end{theorem}
\begin{corollary}[\cite{lam},Theorem 7.55, p. 263]
Let $R$ be any ring that has no infinite set of non-zero orthogonal idempotents. Then the following statements are equivalent:
\begin{itemize}
\item[(1)]$R$ is Baer.
\item[(2)]$R$ is right Rickart.
\end{itemize}
\end{corollary}

\section{Acknowledgements}
Deep thanks to Professor Holly Swisher for editing this paper. Without her support and advising, this paper would not be done.
\\
\\

\bibliographystyle{plain}
\bibliography{mybib}

\end{document}